\documentclass[a4j,12pt]{article}
\usepackage{amsmath}
\usepackage{amssymb}
\usepackage{amsthm}
\usepackage{enumerate}
\newtheorem{theorem}{Theorem}[section]

\newtheorem{Proposition}{Proposition}[section]
\newtheorem{Lemma}{Lemma}[section]
\newtheorem{Corollary}{Corollary}[section]

\newtheorem{Remark}{Remark}[section]

\begin{document}
\title{An ordinary cyclotomic function field}
\author{By D. Shiomi}
\date{}
\maketitle
\section{Introduction}
\quad 
Let $\mathbb{F}_q$ be the field with $q$ elements of characteristic $p$.
Let $k=\mathbb{F}_q(T)$ be the rational function field over $\mathbb{F}_q$,
and $A=\mathbb{F}_q[T]$ the associated polynomial ring.
Let $m\in A$ be a monic polynomial.
Let $K_m$, $K_m^+$ be the $m$-th cyclotomic function field, and
its maximal real subfield (see subsection 2.1).
The aim of this paper is to study the structure
 of the Jacobians of $K_m$, $K_m^+$.
 
%
%
For a global function field  $K$ over ${\mathbb{F}}_q$,
we denote by  $J_K$  the Jacobian of $K\bar{\mathbb{F}}_q$,
where $\bar{\mathbb{F}}_q$ is an algebraic closure of ${\mathbb{F}}_q$.
For a prime $l$,  
it is well-known that the $l$-primary subgroup $J_K(l)$ of $J_K$ is isomorphic 
to the following group
\begin{eqnarray*}
J_K(l) \simeq
\left\{
\begin{array}{ll}
\bigoplus_{i=1}^{2g_K}\mathbb{Q}_l/\mathbb{Z}_l  \;\;\;\;\;\; \text{if $l\neq p$} ,\\ \\
\bigoplus_{i=1}^{\lambda_K}\mathbb{Q}_p/\mathbb{Z}_p  \;\;\;\;\; \text{if $l=p$},
\end{array}
\right.
\end{eqnarray*}
where $g_K$ is the genus of $K$, and $\lambda_K$ is 
called the Hasse-Witt invariant of $K$.
In general, $\lambda_K$ satisfies with
$0\le \lambda_K \le g_K$. 
In particular, we shall call $K$ supersingular if $\lambda_K=0$, and
ordinary if $\lambda_K=g_K$.
For more details of the Jacobian, see \cite{ro2}, \cite{mi}.

%
%
Let $g_m$, $g_m^+$ be the genuses of $K_m$, $K_m^+$,
respectively.
Kida-Murabayashi  
 gave explicit formulas for $g_m$, $g_m^+$ 
 for all monic polynomial $m$ (cf. \cite{km}). 
Hence we obtain the $l$-ranks ($l \neq p$) of $J_{K_m}$, and $J_{K_m^+}$.
      
%
%
%
%
On the other hand,
it is more difficult problem to construct
an explicit formula for Hasse-Witt invariants. 
Let $\lambda_m$, $\lambda_m^+$ be the Hasse-Witt invariants
 of $K_m$, $K_m^+$, respectively.
In the previous paper \cite{sh2}, 
the author completely determined  $m \in A$ satisfying $\lambda_m=0$
(and $\lambda_m^+=0$). 

%
%
%
%
In this paper, we shall consider the ordinary case.
Assume that  $m \in A$ is a monic irreducible polynomial of degree $d$.
 We set 
 \[
 s_i(n)=\sum_{a \in A(i)} a^n,
 \]
 where $A(i)$ is the set of 
 monic polynomials of degree $i$.    
For $1 \le n \le q^d-2$, we define $B_n(u)$ as follows
\begin{eqnarray}
B_n(u)=
\left\{
\begin{array}{ll}
\sum_{i=0}^{d-2}
\Bigl(\sum_{j=0}^{i}s_j(n) \Bigr)u^i &  \text {if }  n \equiv 0 \mod q-1,\\\\
\sum_{i=0}^{d-1}s_i(n)u^i & \text {if }  n \not\equiv 0 \mod q-1.
\end{array}
\right.
\end{eqnarray}
 Let $\mathcal{R}_m=A/mA$, and  $\bar{f}(u) \in \mathcal{R}_m[u]$
 be the reduction of  $f(u) \in A[u]$ modulo $m$.
Now we state our main result in this paper.
%
%
%
%
\begin{theorem}
Let $m \in A$ be a monic irreducible polynomial 
of degee $d$.
Then we have the following results.
\begin{enumerate} 
\item 
 $K_m$ is ordinary if and only if
\begin{eqnarray}
\deg \bar{B}_n(u)=
\left\{
\begin{array}{ll}
\Bigl[\frac{l(n)}{q-1}\Bigr]-1  & \text{if }  n \equiv 0 \mod q-1,\\ \\
\Bigl[\frac{l(n)}{q-1}\Bigr] & \text {if }  n \not\equiv 0 \mod q-1
\end{array}
\right.
\label{eq2}
\end{eqnarray}
for all $1 \le n \le q^d-2$.
\item 
 $K_m^+$ is ordinary if and only if
\begin{eqnarray}
\deg \bar{B}_n(u)=\Bigl[\frac{l(n)}{q-1}\Bigr]-1
\end{eqnarray}
for all $1 \le n \le q^d-2$ $(n \equiv 0 \mod q-1)$.
\label{eq3}
\end{enumerate}
Here $[x]$ is the maximal integer satisfying $[x] \le x$,
and $l(n)=a_0+a_1+\cdots +a_{d-1}$ if $n=a_0+a_1q +
\cdots +a_{d-1}q^{d-1} \;
(  0 \le a_i \le q-1)$.  
\label{th11}
\end{theorem}
%
%
%
%
Assume that $q \neq p$.
By using Theorem \ref{th11}, we will completely
determine a monic irreducible polynomial $m$
such that $K_m$ is ordinary (see Corollary \ref{co31}).
On the other hand, in the case $q=p$, 
it is more difficult problem to determine such $m$.
In section 4, we shall give some examples of ordinary cyclotomic 
function fields.
%
%
%
%
\begin{Remark}
{\rm 
The above polynomial $B_n(u)$ is closely related to 
characteristic $p$ zeta function (cf. \cite{go2}).
}
\end{Remark}
%
%
%
%
%
%
%
%
%
%
\section{Preparations}
%
%
%
%
%
%
%
%
\subsection{Cyclotomic function fields}
\quad In this subsection, we shall provide basic facts about 
cyclotomic function fields. For details, see \cite{ha}, \cite{ro2}, \cite{go2}.

Let $\bar{k}$ be an algebraic closure of $k$. 
For $x \in \bar{k}$ and $m \in A$,  we define the following action
\begin{eqnarray*}
m * x= m(\varphi+\mu)(x),
\end{eqnarray*}
where $\varphi$, $\mu$ are $\mathbb{F}_q$-linear isomorphisms of
 $\bar{k}$ defined by $\varphi:x \mapsto x^q$, and 
$\mu:x\mapsto Tx$, respectively.
 By this action, $\bar{k}$ becomes $A$-module. 
 This $A$-module is called the Carlitz module.
 For a monic polynomial $m \in A$, we set
 \[
 \Lambda_m=\{ x \in \bar{k} \; : \; m*x=0 \}.
  \]
Let $K_m=k(\Lambda_m)$, which is called the
$m$-th cyclotomic function field.
One shows that $K_m/k$ is a Galois extension, and
have the group isomorphism
%
%
%
%
\begin{eqnarray}
\text{Gal}(K_m/k)\simeq (A/mA)^{\times},
\label{eq4}
\end{eqnarray}
where $\text{Gal}(K_m/k)$ is the Galois group of $K_m/k$.
We regard $\mathbb{F}_q^{\times} \subseteq (A/mA)^{\times}$,
and let $K_m^+$ be the intermediate field of $K_m/k$
corresponding to $\mathbb{F}_q^{\times}$.
The field $K_m^+$ is called the maximal real subfield of $K_m$.
Let $P_{\infty}$ be the prime of $k$ with the valuation 
$\text{ord}_{\infty}$ satisfying $\text{ord}_{\infty}(1/T)=1$.
Then $P_{\infty}$ splits completely in $K_m^+/k$, 
and any prime of $K_m^+$ 
over $P_{\infty}$ is totally ramified in $K_m/K_m^+$.
Hence we have 
\[
K_m^+=k_{\infty} \cap K_m,
\]
where $k_{\infty} $ is the associated completion of $k$ by $P_{\infty}$.

%
%
%
%
%
%
%
%
\subsection{Zeta functions}
\quad 
In this subsection, we shall study the zeta function of
cyclotomic function fields. For more references, see \cite{gr}, \cite{ro2}. 

For a global function field $K$ over $\mathbb{F}_q$, 
we define the zeta function of $K$ by 
\begin{eqnarray*}
\zeta(s,K)=\prod_{\mathcal{P}:\text{\rm prime}} \Bigl(1-\frac{1}
{{\mathcal{N} \mathcal{P}}^s}\Bigr)^{-1},
\end{eqnarray*}
where $\mathcal{P}$ runs through all primes of $K$, 
and $\mathcal{N} \mathcal{P}$ is the number of elements of 
the reduce class field of  $\mathcal{P}$.
Then $\zeta(s,K)$ converges absolutely for $\text{\rm Re}(s)>1$.
%
%
%
%
%
%
%
\begin{theorem}
Let $g_K$ be the genus of $K$.
Then there is a polynomial
$Z_K(u) \in \mathbb{Z}[u]$ of degree $2g_K$ satisfying
\begin{eqnarray*}
\zeta(s,K)=\frac{Z_K(q^{-s})}{(1-q^{-s})(1-q^{1-s})}.
\end{eqnarray*}
\label{th21}
\end{theorem}
%
%
%
%

Now we focus on the cyclotomic function field case.
Let $m \in A$ be a monic polynomial of degree $d$. 
Let $\zeta(s,K_m)$, $\zeta(s,K_m^+)$ be zeta functions of $K_m$,
and $K_m^+$, respectively. By  Theorem \ref{th21}, there are polynomials
$Z_m(u)$, and $Z_m^{(+)}(u)$ such that
%
%
%
%
\begin{eqnarray}
\zeta(s,K_m)=\frac{Z_m(q^{-s})}{(1-q^{-s})(1-q^{1-s})},\label{eq5}\\ \notag\\
\zeta(s,K_m^+)=\frac{Z_m^{(+)}(q^{-s})}{(1-q^{-s})(1-q^{1-s})}. \label{eq6}
\end{eqnarray}
Let $X_m$ be the group of primitive Dirichlet characters modulo $m$,
and $X_m^+$ is the subgroup of $X_m$ defined by
\[
X_m^+=
\{ \chi \in X_m : \chi(a)=1 \text{ for all } a \in \mathbb{F}_q^{\times} \} .
\]
By the same arguments in subsection 2.2 in \cite{sh1}, we have
%
%
%
%
\begin{eqnarray}
\zeta(s,K_m)=
\Bigl\{ \prod_{\chi \in X_m}L(s,\chi)  \Bigr\} (1-q^{-s})^{-[K_m^+:k]},\label{eq7}\\  \notag \\
\zeta(s,K_m^+)=
\Bigl\{ \prod_{\chi \in X_m^+}L(s,\chi)  \Bigr\} (1-q^{-s})^{-[K_m^+:k]}\label{eq8}.
\end{eqnarray}
Here an $L$-function $L(s, \chi)$  is defined by
\[
L(s,\chi)=\sum_{a:monic}\frac{\chi(a)}{N(a)^s}, 
\]
where $a$ runs through all monic polynomials of $A$, and
$N(a)=q^{\deg a}$.  
Let  $\chi_0$ be the trivial character.
We can check that
\begin{eqnarray*}
L(s,\chi)=
\left\{
\begin{array}{lll}
1/(1-q^{1-s})\;\;\;\;\;\;\;\;\;\text{if } \chi =\chi_0,\\ \\
\sum_{i=0}^{d-1}s_i(\chi)q^{-si}\;\;\; otherwise,
\end{array}
\right.
\end{eqnarray*}
where $s_i(\chi)=\sum_{a:monic \atop \deg(a)=i}\chi(a)$ for  $i=0, 1,...,d-1$.
We set
\begin{eqnarray*}
\Phi_{\chi}(u)=
\left\{
\begin{array}{lll}
\Bigl(\sum_{i=0}^{d-1}s_i(\chi)u^{i} \Bigr)/(1-u)\;\;\;\;\;\;\;\;\text{if }
 \chi \in X_m^+ \; \backslash \; \{\chi_0\},\\ \\
\sum_{i=0}^{d-1}s_i(\chi)u^{i}\;\;\;\;\;\;\;\;\;\;\;\;\;\;\;\;\;\;\;  \;\;\;\;\;  \;\; \text{if } 
\chi \in X_m^-,
\end{array}
\right.
\end{eqnarray*}
where $X_m^-=X_m \; \backslash \;X_m^+$.
From equations (\ref{eq5}) (\ref{eq6}) (\ref{eq7}) (\ref{eq8}),
 we obtain the following result.
%
%
%
%
%
%
\begin{Proposition}
\begin{eqnarray}
(1) \;\;Z_m(u)&=&
 \prod_{\chi \in X_m \atop \chi \neq \chi_0} \Phi_{\chi}(u), \\ \notag \\
(2) \;Z_m^{(+)}(u)&=&
 \prod_{\chi \in X_m^{+} \atop \chi \neq \chi_0} \Phi_{\chi}(u).
  \text{\quad \quad \quad \quad \quad \quad \quad \quad \quad \quad \quad \quad \quad \quad}.
 \end{eqnarray}
 \label{pr21}
 \end{Proposition}
%
%
%
%
\begin{Remark}
{\rm 
Assume that $\chi \in X_m^+ \; \backslash \; \{\chi_0\}$.
Noting that  $\sum_{i=0}^{d-1}s_i(\chi)=0$, 
we have
%
%
%
%
\begin{eqnarray}
\Phi_{\chi}(u)=\sum_{i=0}^{d-2}\; \Bigl(\sum_{j=0}^{i} s_j(\chi) \; \Bigr)u^{i}. \label{eq11}
\end{eqnarray}
In particular, $\Phi_{\chi}(u)$ is a polynomial.
}
\end{Remark}

 %
 %
 %
 %
 %
 %
 %
 %
 %
 %
\subsection{The Hasse-Witt invarinat}
\quad
Our goal in this subsection is to express $\lambda_m$ 
and $\lambda_m^+$ in terms of $B_n(u)$.
To do this, 
we will study a relation between $B_n(u)$ and $Z_m(u)$
 (and ${Z}^{(+)}_m(u)$).
 For more information, see chapter 8 of \cite{go2}.

Let $m \in A$ be a monic irreducible polynomial of degree $d$.
We denote the $p$-adic field by $\mathbb{Q}_p$.
Fix an algebraic closure $\bar{\mathbb{Q}}$ of $\mathbb{Q}$, 
an algebraic closure $\bar{\mathbb{Q}}_p$ of $\mathbb{Q}_p$,
and an embedding $\sigma:\bar{\mathbb{Q}} \rightarrow \bar{\mathbb{Q}}_p$.
By this embedding, we regard  $\bar{\mathbb{Q}} \subseteq \bar{\mathbb{Q}}_p$.
Let  $\text{ord}_p$ the $p$-adic valuation of $\bar{\mathbb{Q}}_p$
 with $\text{ord}_p(p)=1$. We set
 %
 %
 \[
 M=\mathbb{Q}_p(W),
 \]
 where $W$ is the group of $(p^{de}-1)$-th roots of unity
 (we assume $q=p^e$).
Let $\mathcal{O}_M$ be the valuation ring of $M$. 
Since $M/{\mathbb{Q}}_p$ is unramified, 
the residue class field $\mathcal{F}_M=\mathcal{O}_M/p \mathcal{O}_M$
 consists of  $p^{de}$ elements.
 %
 %
 We notice that
  the image of $\chi \in X_m$ is contained in $\mathcal{O}_M$.
Hence we see that
 \begin{eqnarray*}
\Phi_{\chi}(u) \in \mathcal{O}_M[u] 
\;\;\;(\text{ for} \;\; \chi \in X_m\; \backslash\; \{\chi_0\} \;).
\end{eqnarray*}

%
%
Notice that $\mathcal{R}_m$ and $\mathcal{F}_M$ are 
finite fields with same cardinality.
Hence $\mathcal{R}_m$ is isomorphic to $\mathcal{F}_M$, and
fix an isomorphism $\phi:\mathcal{R}_m \rightarrow \mathcal{F}_M$.
This map derives the group isomorphism 
$\phi_{0}:(A/mA)^{\times} \rightarrow \mathcal{F}_M^{\times}$, 
and the ring isomorphism 
$\phi_{*}:\mathcal{R}_m[u] \rightarrow \mathcal{F}_M[u]$.
Since $p$ is prime to $ {}^{\#}W$ ($=$ the cardinality of $W$), 
we have the following isomorphism
\begin{eqnarray*}
\psi: W \longrightarrow \mathcal{F}_M^{\times}\;\;(\;\zeta \rightarrow
 \zeta \mod p\mathcal{O}_M\;).
\end{eqnarray*}
Put $\omega=\psi^{-1} \circ \phi_0$. 
Then $\omega$ is a generator of $X_m$. Hence we have
%
%
%
%
\begin{eqnarray*}
X_m=\{ \omega^n\; |\; n=0,\;1,\;2,..., \;q^{d}-2 \} .
\end{eqnarray*}
We see that $\omega^n \in X_m^+$  if  $n \equiv 0 \mod q-1$, and 
$\omega^n \in X_m^-$ if $n \not\equiv 0 \mod q-1$.
We notice that
\begin{eqnarray*}
\phi(a^n \mod mA) \equiv \omega^n(a \mod mA) \mod p\mathcal{O}_p
\end{eqnarray*}
for $a \in A \; (0 \le \deg(a) <d)$, and  $n=0,1,..., q^d-2$.
Hence, by the definition of $B_n(u)$, we obtain
\begin{eqnarray*}
\phi_*(\bar{B}_n(u)) =\bar{\Phi}_{\omega^n}(u),
\end{eqnarray*}
where $\bar{\Phi}_{\chi}(u)$ is the reduction of $\Phi_{\chi}(u)$ modulo 
$p \mathcal{O}_M$.
From Proposition \ref{pr21}, we obtain the following results.
%
%
%
%
\begin{Proposition}
\quad 
\begin{eqnarray} 
&&\text{\rm (1) }\phi_*\Bigl(\prod_{n=1}^{q^d-2}\bar{B}_n(u)\Bigr) =\bar{Z}_m(u), \label{eq12}
\;\;\;\;\;\;\;\;\;\;\;\;\;\;\;\;\;\;\;\;\;\;\;\;\;\;\;\;\;\;\;\;\;\;\;\;\;\;\;\;\;\;\;\;\;\;\;\;\;\;\;\;\;\;\;\;
\;\;\;\;\;\;\;\;\;\;\\
&&\text{\rm (2) } \phi_*\Bigl(\prod_{n=1 \atop n \equiv 0 \mod q-1}^{q^d-2} 
\bar{B}_n(u) \Bigr) =\bar{Z}^{(+)}_m(u).\label{eq 13}
\end{eqnarray}
\label{pr22}
 \end{Proposition} 
 
%
%
%
%
Proposition \ref{pr22} leads the following relation
between $\lambda_m$(or $\lambda_m^+$) and $B_n(u)$.
\begin{Corollary}
\begin{eqnarray} 
&&\text{\rm (1) }
\lambda_m=\sum_{n=1}^{q^d-2}\deg \bar{B}_n(u),
\;\;\;\;\;\;\;\;\;\;\;\;\;\;\;\;\;\;\;\;\;\;\;\;\;\;\;\;\;\;\;\;\;\;\;\;\;\;\;\;\;\;\;\;\;\;
\;\;\;\;\;\;\;\;\;\;\;\;\;\;\;\;\;\;\;\;\;  \label{eq14}\\ \notag\\
&&\text{\rm (2) }
 \lambda_m^+=\sum_{n=1 \atop t \equiv 0 \mod q-1}^{q^d-2}
 \deg \bar{B}_n(u). \label{eq15}
\end{eqnarray}
\label{co21}
\end{Corollary}
\begin{proof}
By Proposition 11.20 in \cite{ro2}, we have 
\[
\lambda_m=\deg \bar{Z}_m(u),\;\;\;\;
\lambda_m^+=\deg \bar{Z}_m^{(+)}(u).
\]
Hence we obtain Corollary \ref{co21} from Proposition \ref{pr22}.
\end{proof}
%
%
%
%
%
%
%
%
%
%
%
%
\subsection{Degrees of $B_n(u)$}
\quad In this subsection, we shall study the degree of $B_n(u)$.
To see this, we review some results of Gekeler \cite{ge}.

Fix an integer $d \ge 0$.
For 
$
n=a_0+a_1q+\cdots+a_{d-1}q^{d-1} \;\;(0 \le a_i \le q-1),
$
we define $e_i \; (1 \le i \le l(n))$ as follows:
%
%
\begin{eqnarray*}
n=\sum_{i=1}^{l(n)} q^{e_i} \;\;\;(0 \le e_i \le e_{i+1},\;  e_i<e_{i+q-1}).
\end{eqnarray*}
(Recall that $l(n)=a_0+a_1+\cdots+a_d$).
We set
\begin{eqnarray*}
\rho(n)=
\left\{
\begin{array}{ll}
-\infty  & \text{if } l(n)<q-1,\\\\
 n-\sum_{i=1}^{q-1}q^{e_i} & \text{Otherwise}. 
\end{array}
\right.
\end{eqnarray*}
Moreover
$\rho(-\infty)=-\infty$, 
$\rho^{(0)}(n)=n$, and 
$\rho^{(i)}=\rho^{(i-1)} \circ \rho$.
%
%
%
%
We also put $\deg 0 =- \infty$.
Then Gekeler showed the following result.
\begin{Proposition}(cf. Proposition 2.11 in \cite{ge})
\begin{eqnarray*}
\deg(s_i(n)) \le \rho^{(1)}(n) + \rho^{(2)}(n) \cdots +\rho^{(i)}(n).
\end{eqnarray*}
Moreover, the equality holds if q=p(:prime).
\label{pr23}
\end{Proposition}

In particular, we have the following results.
%
%
%
%
\begin{Corollary}
If $l(n)/(q-1) <i$, then $s_i(n)=0$.
Assume that $q=p$. Then $l(n)/(p-1) <i$ if and only if $s_i(n)=0$. 
\label{co22}
\end{Corollary}

%
%
Next we set
\begin{eqnarray*}
C_n(u)=\sum_{i=0}^{\infty}s_i(n)u^i.
\end{eqnarray*}
From Corollary \ref{co22}, we see that $C_n(u) \in A[u]$.
Moreover,  we have the following result.
%
%
%
%
\begin{Lemma}
$\deg C_n(u) \le \Bigl[ \frac{l(n)}{q-1} \Bigr]$.
The equality holds if $q=p$.
\label{le21}
\end{Lemma}
\begin{proof}
This follows from Corollary \ref{co22}.
\end{proof}
%
%
%
%
\begin{Lemma}
If $1 \le n \le q^d-2 \; (n \equiv 0 \mod q-1)$,
then $C_n(1)=0$.
\label{le22}
\end{Lemma}
\begin{proof}
This follows from Lemma 6.1 in \cite{ge}
\end{proof}
From Lemma \ref{le22}, we obtain
\begin{eqnarray}
B_n(u)=
\left\{
\begin{array}{ll}
C_n(u)/(1-u)&  \text { if }  n \equiv 0 \mod q-1,\\\\
C_n(u) & \text{ if } n \not\equiv 0 \mod q-1
\end{array}
\right.
\label{eq16}
\end{eqnarray}
for $1 \le n \le q^d-2$.
From equation (\ref{eq16}),  we see that  $B_n(u)$ 
is only depend on $n$ ( independent on the choice of $d$).
%
%
%
%
\begin{Proposition}
\begin{eqnarray}
\begin{array}{ll}
(1) \;\;\; \deg B_n(u) \le \Bigl[ \frac{l(n)}{q-1} \Bigr]-1 & \;\; \text {if }  n \equiv 0 \mod q-1,\\ \\
(2)\;\;\; \deg B_n(u) \le \Bigl[ \frac{l(n)}{q-1} \Bigr] & \;\; \text {if }  n \not\equiv 0 \mod q-1.
\end{array}
\end{eqnarray}
In particular, equalities hold if $q=p$.
\label{pr24}
\end{Proposition}
\begin{proof}
This follows from Lemma \ref{le21} .
\end{proof}
%
%
%
%
%
%
\section{A proof of Theorem \ref{th11}}
\quad  Our goal in this section is to prove Theorem \ref{th11}.
To do this,  we first show the following lemma.
%
%
%
%
\begin{Lemma}
For a positive integer $d$, we have
\begin{eqnarray}
\text{(1) } \sum_{n=1 \atop n \equiv 0 \mod q-1}^{q^d-2}
\Bigl[ \frac{l(n)}{q-1} \Bigr]=\frac{d}{2}\Bigl(\frac{q^d-1}{q-1}-1\Bigr), \label{eq18}
\;\;\;\;\;\;\;\;\;\;\;\;\;\;\;\;\;\;\;\;\;\;\;\;\;\;\;\;\;\;\;\;\;\;\;\;\;\;\;\;
\end{eqnarray}
\begin{eqnarray}
\text{(2) } \sum_{n=1 \atop n \not\equiv 0 \mod q-1}^{q^d-2}
\Bigl[ \frac{l(n)}{q-1} \Bigr]=\frac{(d-1)(q-2)(q^d-1)}{2(q-1)}.\label{eq19}
\;\;\;\;\;\;\;\;\;\;\;\;\;\;\;\;\;\;\;\;\;\;\;\;\;\;
\end{eqnarray}
\label{le31}
\end{Lemma}
%
%
%
%
\begin{proof}
We can check that
\begin{eqnarray*}
l(n)+l(q^d-1-n)=(q-1)d
\end{eqnarray*} 
for $1 \le n \le q^d-2$.
Assume that $n \equiv 0 \mod q-1$.
Since $l(n) \equiv l(q^d-1-n) \equiv 0 \mod q-1$,
we have 
\begin{eqnarray*}
\Bigl[ \frac{l(n)}{q-1} \Bigr]+\Bigl[ \frac{l(q^d-1-n)}{q-1} \Bigr]=d.
\end{eqnarray*}
Therefore,
\begin{eqnarray*}
\sum_{n=1 \atop n \equiv 0 \mod q-1}^{q^d-2}
\Bigl{\{}
\Bigl[ \frac{l(n)}{q-1} \Bigr]+\Bigl[ \frac{l(q^d-1-n)}{q-1} \Bigr] \Bigr{\}}
=d \Bigl(\frac{q^d-1}{q-1}-1 \Bigr).
\end{eqnarray*}
This leads equation (\ref{eq18}).
%
%
%
%
%
%
Next we assume that $n  \not\equiv 0 \mod q-1$.
Then
\begin{eqnarray*}
\Bigl[ \frac{l(n)}{q-1} \Bigr]+\Bigl[ \frac{l(q^d-1-n)}{q-1} \Bigr]=d-1.
\end{eqnarray*}
Therefore,
\begin{eqnarray*}
\sum_{n=1 \atop n \not\equiv 0 \mod q-1}^{q^d-2}
\Bigl{\{}
\Bigl[ \frac{l(n)}{q-1} \Bigr]+\Bigl[ \frac{l(q^d-1-n)}{q-1} \Bigr]
\Bigr{\}}
=\frac{{(d-1)(q-2)(q^d-1)}}{(q-1)}.
\end{eqnarray*}
Hence we obtain equation (\ref{eq19}).
\end{proof}
%
%
%
%
%
%
%
%
%
%
%
%
%
Now we give the proof of Theorem \ref{th11}.
\begin{proof}
One shows that $g_m$, $g_m^+$ can be
calculated as follows
%
%
\begin{eqnarray}
2g_m&=&(dq-d-q)\Bigl(\frac{q^d-1}{q-1}\Bigr)-(d-2) \label{eqA}, \\ \notag \\
2g_m^+&=&(d-2)\Bigl(\frac{q^d-1}{q-1}-1\Bigr) \label{eqB}
\end{eqnarray}
(cf. [K-M]).
By comparing with Lemma \ref{le31}, we obtain
%
%
\begin{eqnarray}
g_m&=&\sum_{n=1 \atop n \equiv 0 \mod q-1}^{q^d-2}
\Bigl( \Bigl[ \frac{l(n)}{q-1} \Bigr]-1\Bigr) +
\sum_{n=1 \atop n \not\equiv 0 \mod q-1}^{q^d-2}
\Bigl[ \frac{l(n)}{q-1} \Bigr], \label{eq20} \\ \notag\\
g_m^+&=&\sum_{n=1 \atop n \equiv 0 \mod q-1}^{q^2-2}
\Bigl(\Bigl[ \frac{l(n)}{q-1} \Bigr]-1 \Bigr) \label{eq21}.
\end{eqnarray}
%
%
%

First we assume that $\lambda_m=g_m$. 
Then, by Corollary \ref{co21} and Proposition \ref{pr24}, and equation (\ref{eq20}),
we see that equation (\ref{eq2}) holds. 
Conversely, we assume that equation (\ref{eq2}) holds. 
Then, by Corollary \ref{co21} and equation (\ref{eq20}), 
we obtain $\lambda_m=g_m$.
This complete the proof of the part 1 of Theorem \ref{th11}.

By the same arguments, we can prove  the part 2 of Theorem \ref{th11}.
\end{proof}

\begin{Remark}
{\rm
%
%
%
%
%
%
From the proof of Theorem \ref{th11},
 we have the following results.
\begin{enumerate}
\item
If $K_m$ is ordinary, then 
\begin{eqnarray*}
\deg \bar{B}_n(u)=\deg B_n(u)=
\left\{
\begin{array}{ll}
\Bigl[\frac{l(n)}{q-1}\Bigr]-1  & \text{if }  n \equiv 0 \mod q-1,\\ \\
\Bigl[\frac{l(n)}{q-1}\Bigr] & \text {if }  n \not\equiv 0 \mod q-1
\end{array}
\right.
\end{eqnarray*}
for all $1 \le n \le q^d-2$.
\item
If $K_m^+$ is ordinary,
then
\begin{eqnarray*}
\deg \bar{B}_n(u)=\deg B_n(u)
=\Bigl[\frac{l(n)}{q-1}\Bigr]-1 
\end{eqnarray*}
for all $1 \le n \le q^d-2 \;\;(n \equiv 0 \mod q-1)$.
\end{enumerate}
}
\end{Remark}
By using Theorem \ref{th11}, 
we determine all ordinary cyclotomic function field
in the case of $q \neq p$.
%
%
%
%
%
%
\begin{Corollary}
We assume that $q \neq p$. 
Let $m$ be a monic irreducible polynomial.
Then we have the following results.
\begin{enumerate}
\item $K_m$ is ordinary if and only if $\deg m = 1$.
\item $K_m^+$ is ordinary if and only if $\deg m \le 2$.
\end{enumerate}
\label{co31}
\end{Corollary}
 \begin{proof}
 %
 %
 First we show the assertion 1.
Assume that $\deg m = 1$. Then we obtain $g_m=0$ by
equation (\ref{eqA}).  Hence $K_m$ is ordinary.
Next, we put $n=(q-p)+pq$.
Then $l(n)=q \not\equiv 0 \mod q-1$.
By Corollary 3.14 in \cite{ge}, we have
\begin{eqnarray*}
s_1(n)=-\left(
\begin{array}{c}
p \\
p-1
\end{array}
\right)
(T^p-T)=0.
\end{eqnarray*}
Hence $B_n(u)=1$. 
Notice that $\deg B_n(u) <[\frac{l(n)}{q-1}]$. 
It follows that
 $K_m$ is not ordinary if $\deg m \ge 2$.
This leads the assertion 1 of Corollary \ref{co31}.

%
%
%
Secondly, we will show the assertion 2 of Corollary \ref{co31}.
By equation (\ref{eqB}),
we see that $K_m^+$ is ordinary if $\deg m \le 2$.
Next, 
we put $n=p+(q-p)q+(q-2)q^2$, and
$n_0=n/p=1+(q-q/p-1)q+(q/p-1)q^2$.
Then we have $l(n)=2(q-1)$, and $l(n_0)=q-1$.
By Proposition \ref{pr24},  we have $1+s_{1}(n_0)=0$.
Noting that 
\begin{eqnarray*}
1+s_1(n)=(1+s_1(n_0))^p=0,
\end{eqnarray*}
we have $B_n(u)=1$.
Hence $\deg B_n(u)<[\frac{l(n)}{q-1}]-1$.
It follows that
$K_m^+$ is not ordinary if $\deg m \ge 3$.
This leads the assertion 2 of Corollary \ref{co31}.
 \end{proof}
 
 The above corollary  is not true in the case $q=p$.
 We will see this in the next section.

%
%
%
%
%
%
%
%
\section{Some examples of ordinary cyclotomic function field}
\quad 
In this section, we assume $q=p$.
As an application of Theorem \ref{th11}, we shall construct some
 examples of ordinary cyclotomic function fields.
%
%
%
%
\begin{Proposition}
Assume $m \in A$ is a monic irreducible polynomial 
of degree two. Then $K_m^+$, and $K_m$ are ordinary.
\label{pr41}
\end{Proposition}
%
%
\begin{proof}
From equation (\ref{eqA}), we have $g_m^+=0$.
Hence $K_m^+$ is ordinary.

Next we will show that $K_m$ is ordinary.
To see this, we shall see that equation (2) holds.
We first consider the case $l(n) \le p-1$.
By Proposition \ref{pr24}, we have $B_n(u)=1$.
Hence equation (\ref{eq2}) holds in this case.

Secondly, we consider the case $p \le l(n) < 2(p-1)$.
Noting that $n \not\equiv 0 \mod p-1$, we obtain
\begin{eqnarray*}
B_n(u)=1+s_1(n)u.
\end{eqnarray*}
Here we put $n=a+bp\;(0 \le a,\;b \le p-1)$.
Then Gekeler  showed
\begin{eqnarray*}
s_1(n)=-\left(
\begin{array}{c}
b \\
p-1-a
\end{array}
\right)
(T^p-T)^{a+b-(p-1)}
\end{eqnarray*}
(cf. Corollary 3.14 in \cite{ge}).
Hence $s_1(n) \not\equiv 0 \mod m$.
Therfore equation (\ref{eq2}) holds in this case.
This complete the proof of Proposition \ref{pr41}.
\end{proof}
%
%
%
%
\begin{Proposition}
Assume that $m \in A$ is a monic irreducible polynomial 
of degree three. Then $K_m^+$ is ordinary.
\label{pr42}
\end{Proposition}
%
%
%
%

\begin{proof}
Fix an integer $n$ such that  
$1 \le n \le p^3-2 \; (n \equiv 0 \mod p-1)$.
Then we have
\begin{eqnarray*}
B_n(u)=1+f_n(T)u,
\end{eqnarray*}
where $f_n(T)$ is defined by
\begin{eqnarray*}
f_n(T)=1+s_1(n)=1+\sum_{\alpha \in \mathbb{F}_q} (T +\alpha)^n.
\end{eqnarray*}
We notice that $l(n)=p-1$ or $2(p-1)$.
First, we assume that $l(n)=p-1$.
Then, by Proposition \ref{pr24}, we have $B_n(u)=1$.
Hence equation (3) holds in this case.

%
%
%
%
Secondly, we consider the case $l(n)=2(p-1)$.
From Proposition \ref{pr24},  we see that  $f_n(T) \neq 0$.
Assume that $f_n(T) \equiv 0 \mod m$.
Let $\omega$ be a root of $m$.
Then $\omega$ is also a root of $f_n(T)$.
We put
%
%
\begin{eqnarray*}
W_1&=&\Bigl\{ a+\frac{b}{\omega+c} : 
a, c \in \mathbb{F}_q, \; b\in \mathbb{F}_q^{\times} \Bigr\},\\ \\
W_2&=&\Bigl\{ a+b\omega : 
a, b \in \mathbb{F}_q \Bigr\}.
\end{eqnarray*}
We can easily check that
(i) $f_n(T+\alpha)=f_n(T) \;\;(\alpha \in \mathbb{F}_q)$,
(ii) $f_n(\alpha T)=f_n(T) \;\;(\alpha \in \mathbb{F}_q^{\times})$,
(iii) $T^nf_n(1/T)=f_n(T)$.
Hence each element of $W_1 \cup W_2$ is also
a root of $f_n(T)$.
Notice that $\omega$ is a root of irreducible polynomial of degree $3$.
Hence 
\begin{eqnarray*}
W_1 \cap W_2 =\phi,\;\;{}^{\#}W_1=p^3-p^2,\;\;{}^{\#}W_2=p^2.
\end{eqnarray*}
Therefore $f_n(T)$ has distinct $p^3$ roots. 
However $\deg f_n(T) \le n \le p^3-2$. 
This is a contradiction.
Therefore $f_n(T) \not\equiv 0 \mod m$.
Hence equation (3) holds in this case.
\end{proof}
%
%
%
%
\begin{Remark}
{\rm
The above result is not true for $K_m$.
In fact, we consider the case $p=3$ and $m=T^3+2T+1$.
Then $g_m=19$, and $\lambda_m=18$.
}
\end{Remark}

%
%
%
%
%

\quad \\
Daisuke Shiomi\\
Graduate School of Mathematics,  Nagoya University\\
Furou-cho, Chikusa-ku, Nagoya 464-8602, Japan\\
Mail: m05019e@math.nagoya-u.ac.jp\\

\end{document}